\documentclass[10pt,reqno]{amsart}
\usepackage{amsmath,amssymb,amsthm,amsfonts,enumerate,hyperref,stmaryrd}
\usepackage{graphicx}
\theoremstyle{plain}
\newtheorem{theorem}{Theorem}[section]

\newtheorem{lemma}[theorem]{Lemma}
\newtheorem{conj}[theorem]{Conjecture}

\theoremstyle{definition}

\theoremstyle{remark}

\numberwithin{equation}{section}

\newcommand{\eps}{\varepsilon}
\newcommand{\abs}[1]{\lvert #1\rvert}
\newcommand{\Lap}{\Delta}
\newcommand{\D}{\nabla}
\newcommand{\CC}{\mathbb{C}}
\newcommand{\RR}{\mathbb{R}}
\newcommand{\Riem}{\mathrm{Rm}}

\newcommand{\vol}{\mathrm{Vol}}
\newcommand{\diag}{\mathrm{diag}}
\newcommand{\Rc}{\mathrm{Rc}}

\newcommand{\Hess}{\mathrm{Hess}}
\newcommand{\tr}{\mathrm{tr}}

\renewcommand{\div}{\mathrm{div}}

\begin{document}

\title{The stability inequality for Ricci-flat cones}

\begin{abstract}
In this article, we thoroughly investigate the stability inequality for Ricci-flat cones. Perhaps most importantly, we prove that the Ricci-flat cone over $\CC P^2$ is stable, showing that the first stable non-flat Ricci-flat cone occurs in the smallest possible dimension. On the other hand, we prove that many other examples of Ricci-flat cones over 4-manifolds are unstable, and that Ricci-flat cones over products of Einstein manifolds and over K\"ahler-Einstein manifolds with $h^{1,1}>1$ are unstable in dimension less than 10. As results of independent interest, our computations indicate that the Page metric and the Chen-LeBrun-Weber metric are unstable Ricci shrinkers. As a final bonus, we give plenty of motivations, and partly confirm a conjecture of Tom Ilmanen relating the $\lambda$-functional, the positive mass theorem and the nonuniqueness of Ricci flow with conical initial data.
\end{abstract}

\author{Stuart Hall}
\address{Department of Applied Computing, The University of Buckingham, UK}
\email{stuart.hall@buckingham.ac.uk}

\author{Robert Haslhofer}
\address{Department of Mathematics, ETH Z\"{u}rich, Switzerland}
\email{robert.haslhofer@math.ethz.ch}

\author{Michael Siepmann}
\address{Department of Mathematics, ETH Z\"{u}rich, Switzerland}
\email{michael.siepmann@math.ethz.ch}

\maketitle

\section{Introduction}

The purpose of this article is to thoroughly investigate the stability of Ricci-flat cones. To start with, we call a Ricci-flat manifold or cone \emph{stable}, if
\begin{equation}\label{stabineq}
\int_M 2\Riem(h,h)\leq \int_M \abs{\nabla h}^2
\end{equation}
for all transverse symmetric 2-tensors $h$ with compact support. In this definition, $\Riem(h,h)=R_{ijkl}h_{ik}h_{jl}$ and transverse means $\D_i h_{ij}=0$.\\

Historical motivations for this notion of stability come from considering the linearization of the Ricci curvature and from computing the second variation of the Einstein-Hilbert functional restricted to transverse-traceless tensors \cite{Be}. A more modern motivation comes from the second variation of Perelman's $\lambda$-functional \cite{Pe}. In the case of closed manifolds this second variation is computed in \cite{CHI}. In the case of non-compact spaces one can use a suitable variant of Perelman's energy functional as explained in \cite{Ha2}.\\

Stability plays an important role in gravitational physics, see e.g. \cite{GPY}, and also is of great numerical relevance. Complementing that, our main motivations for the problem under consideration come from questions concerning the regularity of Ricci-flat spaces, from questions about generic Ricci flow singularities, and from a conjecture of Tom Ilmanen  relating the $\lambda$-functional, the positive mass theorem and the nonuniqueness of Ricci flow with conical initial data. They are all inspired by analogies and disanalogies between minimal surfaces (respectively the mean curvature flow) and Ricci-flat spaces (respectively the Ricci flow), and will be described now.

\subsection{Motivation I: Regularity of Ricci-flat spaces}
To start with the mentioned analogy, the stability inequality for minimal hypersurfaces,
\begin{equation}
\int_\Sigma \left(\abs{A}^2+\Rc(\nu,\nu)\right)\psi^2 \leq \int_\Sigma \abs{\nabla \psi}^2,
\end{equation}
has a long and successful history. In particular, it plays a prominent role in the work of Schoen-Yau on the structure of manifolds with positive scalar curvature and in their proof of the positive mass theorem \cite{SY1,SY2}. By Almgren's big regularity theorem the singular set of a general area-minimizing rectifiable current has codimension at least $2$ \cite{Al}. In the case of \emph{stable} minimal hypersurfaces there is a much better (and in fact much easier to prove) conclusion: the singular set has codimension at least $7$ \cite{SiL}. The number $7$ ultimately comes from the Simons cones \cite{SiJ}, cones with vanishing mean curvature over a product of spheres. In analogy with the Simons cones, as computed some time ago by Tom Ilmanen \cite{Ipc} (see also \cite{GHP}), Ricci-flat cones over a product of spheres are unstable in dimension less than $10$.\\

In the general case of Ricci curvature, there is a deep regularity theory developed by Cheeger, Colding and Tian, see \cite{C} for a very nice survey (see also the recent work of Cheeger-Naber \cite{CN} for very interesting quantitative results). In particular (focusing on the Ricci-flat case here), they proved that the singular set of non-collapsed Gromov-Hausdorff limits of Ricci-flat manifolds with special holonomy has codimension at least $4$ (it is conjectured that this also holds without the special holonomy assumption). Given this general regularity theory, it is an intriguing question whether or not an improved regularity occurs in the case of \emph{stable} Ricci-flat spaces.

\subsection{Motivation II: Generic Ricci flow singularities}
Recently, Colding and Minicozzi proposed a generic mean curvature flow in dimension two \cite{CM1}. The idea, going back to Huisken, is that very complicated singularities can form in two-dimensional mean curvature flow, but all except the cylinders and spheres are unstable and can be perturbed away. Singularities are modeled on shrinkers, and recently Kapouleas-Kleene-M{\o}ller \cite{KKM} constructed examples of shrinkers with arbitrarily high genus.\\

Two-dimensional mean curvature flow shares many similarities with four-dimensional Ricci flow. One reason for this is that the integrals $\int \abs{A}^2d\mu$ and $\int \abs{\Riem}^2dV$ are scale invariant in dimension $2$ and $4$, respectively. These integrals show up naturally in the compactness theorems for shrinkers in \cite{CM2} and \cite{HM} and can be estimated using the Gauss-Bonnet formula in the respective dimension. As hinted at in \cite{CHI}, there is hope that actually all but very few singularity models for four-dimensional Ricci flow are unstable (the known stable examples are $S^4,S^3\times \RR,S^2\times \RR^2,\mathbb{C}P^2$ and their quotients, and presumably the blow-down shrinker of Feldman-Ilmanen-Knopf \cite{FIK} is also stable). This might lead to a theory of generic Ricci flow in dimension four.

\subsection{Motivation III: Ilmanen's conjecture}
In July 2008, Tom Ilmanen conjectured an intriguing relationship for ALE-spaces and Ricci-flat cones between:
\begin{enumerate}
\item lambda not a local maximum
\item failure of positive mass
\item nonuniqueness of Ricci flow
\end{enumerate}
He also told us that part of the problem is to find an appropriate definition of Perelman's $\lambda$-functional in the noncompact case. We will explain all this later. For the moment, let us just mention that the positive mass theorem \cite{SY2, Wi} does not generalize to ALE-spaces \cite{LB}, that Ricci flows coming out of cones have been constructed in \cite{FIK,GK,SS}, and that nonuniqueness caused by unstable cones has been observed for various other geometric heat flows \cite{I,GR}.

\subsection{Heuristics about proving instability}
To prove instability of a Ricci-flat cone, we have to find an almost parallel test variation $h$  (in the sense that the right hand side of (\ref{stabineq}) is small), such that the left hand side is large. To discuss the left hand side, let us diagonalize $h=\diag(\lambda_1,\ldots,\lambda_n)$ in an orthonormal basis $e_1,\ldots,e_n$ and introduce the following curvature condition: The \emph{sectional curvature matrix} $K$ is the matrix with the entries (no sum!)
\begin{equation}
K_{ij}=\Riem(e_i,e_j,e_i,e_j).
\end{equation}
Observe that, $\Riem(h,h)=\lambda^TK\lambda$, so to make the left hand side large, we have to find a large eigenvalue of the sectional curvature matrix. Note that $K$ is a symmetric matrix with vanishing diagonal. Thus, unless $K=0$, there exists at least one positive eigenvalue. Taking into account the interaction with the Hardy-inequality, we thus expect that many Ricci-flat cones in small dimensions are unstable.

\subsection{The results}
Obviously, flat implies stable. Next, in analogy with the Simons cones, we consider Ilmanen's example of Ricci-flat cones over products of spheres or more generally Ricci-flat cones over products of Einstein manifolds:

\begin{theorem}\label{thm1}
A Ricci-flat cone over a product of Einstein manifolds is unstable in dimension less than 10.
\end{theorem}

We also prove:

\begin{theorem}\label{thm2}
A Ricci-flat cone over a K\"ahler-Einstein manifold with Hodge number $h^{1,1}>1$ is unstable in dimension less than 10.
\end{theorem}

Given Theorem \ref{thm1} and Theorem \ref{thm2}, one might guess for a moment that the singular set of stable Ricci-flat spaces has codimension at least $10$. However, we manage to show that this naive guess is very wrong. In fact, we prove that the first stable non-flat Ricci-flat cone occurs in the smallest possible dimension (note that every at most 4-dimensional Ricci-flat cone is flat, since every at most 3-dimensional Einstein metric is automatically a space-form):

\begin{theorem}\label{thm3}
The Ricci-flat cone over $\CC P^2$ is stable.
\end{theorem}

Note that no weak definition of Ricci curvature is available (see however the recent week definition of Ricci bounded below due to Lott-Villani and Sturm \cite{LV,St}),  but the main point of Theorem \ref{thm3} is that the analogy with stable minimal hypersurfaces breaks down already at the level of cones. To rescue the analogy between minimal surfaces and Ricci-flat spaces, one can of course reformulate this last statement, and conclude that Ricci-flat spaces behave more like minimal surfaces of general codimension.\\

The above theorems answer the question posed in Motivation I. Related to Motivation II, and also to obtain a more detailed understanding of the problem from Motivation I, we examine the case of Ricci-flat cones over 4-manifolds in greater detail: Besides the one over $\CC P^2$, the other fundamental examples of Ricci-flat cones over 4-manifolds are the one over $S^2\times S^2$, the ones over $(\CC P^2\sharp p \overline{\CC P^2})\, {(3\leq p\leq 8)}$ with Tian's K\"ahler-Einstein metrics \cite{Ti}, the one over $\CC P^2\sharp \overline{\CC P^2}$ with the Page metric \cite{Pa}, and the one over $\CC P^2\sharp 2 \overline{\CC P^2}$ with the Chen-LeBrun-Weber metric \cite{CLW}. We prove:

\begin{theorem}\label{thm4}
The Ricci-flat cones $C(S^2\times S^2)$, $C(\CC P^2\sharp p \overline{\CC P^2})\, {(3\leq p\leq 8)}$ and $C(\CC P^2\sharp \overline{\CC P^2})$, are all unstable.
\end{theorem}

For the cone over the Chen-LeBrun-Weber metric we carried out a computation that strongly indicates that it is also unstable. However, there is one particular step in our computation that relies on a numerical estimate for extremal K\"ahler metrics using the algorithm of Donaldson-Bunch \cite{DonBun}. Therefore, we state the following as a conjecture and not as a theorem:
\begin{conj}\label{conj1}
The Ricci-flat cone $C(\CC P^2\sharp 2 \overline{\CC P^2})$ is unstable.
\end{conj}

As results of independent interest, our computations have applications for Ricci shrinkers. A Ricci shrinker is given by a complete Riemannian manifold $(M,g)$ and a smooth function $f$ such that
\begin{equation}
\Rc+\Hess f=\tfrac{1}{2\tau}g,\label{shrinkereq}
\end{equation}
for some $\tau>0$. Solutions of equation (\ref{shrinkereq}) correspond to self-similarly shrinking solutions of Hamilton's Ricci flow \cite{Ham82}, and they model the formation of singularities in the Ricci flow (see \cite{Ca} for a recent survey on Ricci solitons). By the formula for the second variation of Perelman's shrinker entropy \cite{Pe,CHI,CZhu,HallM}, the stability inequality for shrinkers is
\begin{equation}
\int_M 2\Riem(h,h)e^{-f}\leq \int_M \abs{\nabla h}^2e^{-f}
\end{equation}
for all symmetric 2-tensors $h$ with compact support satisfying $\div(e^{-f}h)=0$ and $\int_M \tr h\,e^{-f}=0$. In particular, for positive Einstein metrics this is formally the same as (\ref{stabineq}) with the additional requirement that $\int_M \tr h=0$. Our computations prove:
\begin{theorem}\label{thm5}
$\CC P^2\sharp \overline{\CC P^2}$ with the Page metric is an unstable Ricci shrinker.
\end{theorem}
Numerical evidence for this result was already given by Roberta Young in 1983 \cite{Young1983}. Modulo the estimate for extremal  K\"ahler metrics mentioned above, we can also prove:
\begin{conj}\label{conj2}
$\CC P^2\sharp 2 \overline{\CC P^2}$ with the Chen-LeBrun-Weber metric is an unstable Ricci shrinker.
\end{conj}

Theorem \ref{thm5} and Conjecture \ref{conj2} are relevant for developing a theory of generic Ricci flow in dimension four (the instability of $S^2\times S^2$ and $\CC P^2\sharp p \overline{\CC P^2}$ for $3\leq p\leq 8$ has already been observed in \cite{CHI}).\\

Finally, regarding Motivation III we prove a theorem which we informally state as follows:
\begin{theorem}\label{thm6}
The implications $(1)\Leftrightarrow(2)$ and $(3)\Rightarrow(1)$ in Ilmanen's conjecture hold. Regarding $(1)\stackrel{?}{\Rightarrow}(3)$, for some unstable Ricci-flat cones over four-manifolds there do not exist instantaneously smooth Ricci flows coming out of them, but possibly there do exist many singular solutions.
\end{theorem}

To answer $(1)\stackrel{?}{\Rightarrow}(3)$ in full generality, one essentially would have to develop a theory of weak Ricci flow solutions first. In Theorem \ref{thm6} we use a suitable noncompact variant of Perelman's energy functional introduced in \cite{Ha2}. For the detailed statement and further explanations we refer to Section \ref{ilmanenconj}.\\

Finally, it would be very interesting to obtain a better picture about the \emph{dynamical stability} and about the \emph{dynamical instability} of noncompact Ricci-flat spaces under the Ricci flow. On the one hand, one could try to do some explicit computations for Ricci-flat cones with enough symmetry. On the other hand, one could try to generalize the techniques introduced in \cite{Se} and \cite{Ha1} to the noncompact setting.\\

This article is organized as follows: In Section \ref{basiccones} we collect some basic facts about Ricci-flat cones. In Section \ref{secprod} and Section \ref{seck} we prove Theorem \ref{thm1} and Theorem \ref{thm2} respectively. In Section \ref{seccpt} we carefully investigate the important example of the Ricci-flat cone over $\CC P^2$ and prove Theorem \ref{thm3}. In Section \ref{secff} we analyse the other fundamental examples of Ricci-flat cones over four manifolds, in particular the ones over the manifolds with the Page metric and the Chen-LeBrun-Weber metric. The first part of the proof of Theorem \ref{thm4} and Theorem \ref{thm5} is in that section, while the second part of the proof is based on estimates for extremal K\"ahler metrics which we carry out in Section \ref{extremalk}. In the latter two sections, we also give evidence for Conjecture \ref{conj1} and Conjecture \ref{conj2}. Finally, in Section \ref{ilmanenconj} we discuss Ilmanen's conjecture and prove Theorem \ref{thm6}.\\

\textbf{Acknowledgements.} We thank Richard Bamler, Pierre Germain, Nadine Gro{\ss}e, Martin Li, Rafe Mazzeo, Thomas Murphy, Melanie Rupflin, Richard Schoen, Miles Simon, Michael Struwe, and especially Tom Ilmanen and Simon Donaldson for useful discussions and interesting comments. The research of RH and MS has been partly supported by the Swiss National Science Foundation, SH was supported by the EPSRC. RH also thanks the HIM Bonn for hospitality and financial support.

\section{Basic facts about Ricci-flat cones}\label{basiccones}
Let $(M,g)=C(\Sigma,\gamma)=(\mathbb{R}_{+}\times \Sigma,dr^2+r^2\gamma)$ be the Riemannian cone over a closed $(n-1)$-dimensional manifold $(\Sigma,\gamma)$. We write $x^0=r$, and let $(x^1,....,x^{n-1})$ be coordinates on $\Sigma$. The non-vanishing Christoffel symbols are:
\begin{align}
\Gamma(g)_{ij}^{k}=\Gamma(\gamma)_{ij}^{k}\qquad \Gamma(g)_{ij}^{0}=-r\gamma_{ij}\qquad
\Gamma(g)_{i0}^{k}=\Gamma(g)_{0i}^{k}=\tfrac{1}{r}\delta_i^k.
\end{align}
Thus, the non-vanishing components of the Riemann-tensor are
\begin{equation}
R(g)_{ijkl}=r^2(R(\gamma)_{ijkl}-\gamma_{ik}\gamma_{jl}+\gamma_{il}\gamma_{jk}).
\end{equation}
In particular, the cone is Ricci-flat if and only if
\begin{equation}
\Rc(\gamma)=(n-2)\gamma,
\end{equation}
which we will always assume in the following. The basic facts about cones collected here will be used frequently and without further reference in the following parts of the article. Also, we will always assume that all variations are supported away from the tip of the cone (however, this assumption could be relaxed by a standard approximation argument).

\section{Ricci-flat cones over product spaces}\label{secprod}

Suppose $(\Sigma_1,\gamma_1)$ and $(\Sigma_2,\gamma_2)$ are two positive Einstein manifolds of dimension $n_1$ and $n_2$ respectively. Then the associated cone $(M,g)=C(\Sigma_1\times\Sigma_2)$ of dimension $n=n_1+n_2+1$ is indeed Ricci-flat if normalize the Einstein metrics to have Einstein constant $n-2$. We will now prove:

\begin{theorem}\label{produnstabthm}
The Ricci-flat cone $C(\Sigma_1\times \Sigma_2)$ is unstable for $n<10$.
\end{theorem}

\begin{proof} Consider the  variation (geometrically, this corresponds to making one factor larger and the other factor smaller):
\begin{equation}
h=f(r)r^2\left(\tfrac{\gamma_1}{n_1}-\tfrac{\gamma_2}{n_2}\right).
\end{equation}
Note that
\begin{equation}
\nabla_0h_{ij}=\tfrac{f'}{f}h_{ij}\quad \nabla_kh_{0j}=\nabla_kh_{j0}=-\tfrac{1}{r}h_{kj},
\end{equation}
while the other components vanish. Thus $h$ is transverse-traceless (TT), and
\begin{equation}
\abs{\nabla h}^2=\left(\tfrac{1}{n_1}+\tfrac{1}{n_2}\right)\left(f'^2+2\tfrac{f^2}{r^2}\right).\label{normofthegradient2}
\end{equation}
Furthermore, using the notation $(\gamma\odot\gamma)_{ijkl}=\gamma_{ik}\gamma_{jl}-\gamma_{il}\gamma_{jk}$, we compute
\begin{align}
\Riem_g(h,h)&=\tfrac{f^2}{r^2}\big(\Riem_\gamma-\gamma\odot\gamma\big)\big(\tfrac{\gamma_1}{n_1}-\tfrac{\gamma_2}{n_2},\tfrac{\gamma_1}{n_1}-\tfrac{\gamma_2}{n_2}\big)\\
&=\tfrac{f^2}{r^2}\left[\tfrac{1}{n_1^2}(\Riem_{\gamma_1}-\gamma_1\odot\gamma_1)(\gamma_1,\gamma_1)+\tfrac{2}{n_1n_2}\mathrm{tr}_{\gamma_1}(\gamma_1)\mathrm{tr}_{\gamma_2}(\gamma_2)\right.\nonumber\\
&\qquad\quad\left.+\tfrac{1}{n_2^2}(\Riem_{\gamma_2}-\gamma_2\odot\gamma_2)(\gamma_2,\gamma_2)\right]\nonumber\\
&=\tfrac{f^2}{r^2}\left[\tfrac{n_2}{n_1}+2+\tfrac{n_1}{n_2}\right].\nonumber
\end{align}
Putting everything together, we obtain
\begin{equation}
|\nabla h|^2-2\text{Rm}(h,h)=\left(\tfrac{1}{n_1}+\tfrac{1}{n_2}\right)\left(f'^2-2(n-2)\tfrac{f^2}{r^2}\right),
\end{equation}
and thus
\begin{multline}\label{RRQuotient}
\int_{C(\Sigma_1\times \Sigma_2)}\left(|\nabla h|^2-2\text{Rm}(h,h)\right)dV\\
=\left(\tfrac{1}{n_1}+\tfrac{1}{n_2}\right)\vol(\Sigma_1\times \Sigma_2)\int_{0}^{\infty}\left(f'^2-2(n-2)\tfrac{f^2}{r^2}\right)r^{n-1}dr.
\end{multline}
Recall that $C_H=4/(n-2)^2$ is the optimal constant in the Hardy-inequality
\begin{equation}
\int_0^\infty\tfrac{f^2}{r^2}r^{n-1}dr\leq C_H\int_0^\infty f'^2 r^{n-1}dr.
\end{equation}
Thus, the expression in (\ref{RRQuotient}) can become negative if and only if
\begin{equation}
2(n-2)C_H=\tfrac{8}{n-2}>1,
\end{equation}
i.e. if and only if $n<10$. This proves the theorem.
\end{proof}

\section{Ricci-flat cones over K\"{a}hler-Einstein manifolds}\label{seck}

Suppose $(\Sigma,\gamma)$ is K\"{a}hler-Einstein with Einstein constant $n-2$ and, and let $(M,g)=C(\Sigma)$ be the corresponding Ricci-flat cone of real dimension $n$.

\begin{theorem}\label{kaehlerthm}
If $h^{1,1}(\Sigma)>1$, then $C(\Sigma)$ is unstable for $n<10$.
\end{theorem}

\begin{proof}
The assumptions of the theorem imply that there exists a transverse-traceless symmetric 2-tensor $k\neq 0$ on $(\Sigma,\gamma)$ with
\begin{equation}
(\Lap+2\Riem_\gamma)k=2(n-2)k.
\end{equation}
Consider the  variation (this is quite related to the previous section, and geometrically corresponds to making one $(1,1)$-cycle larger and another one smaller):
\begin{equation}
h=f(r)r^2k.
\end{equation}
This variation is TT, and we compute
\begin{align}
&\abs{\nabla h}^2=\left(f'^2+2\tfrac{f^2}{r^2}\right)\abs{k}^2+\tfrac{f^2}{r^2}\abs{\nabla k}^2\\
&2\Riem(h,h)=2\tfrac{f^2}{r^2}\left(\Riem_\gamma(k,k)+\abs{k}^2\right).
\end{align}
Thus
\begin{multline}
\int_{C(\Sigma)}\left(|\nabla h|^2-2\text{Rm}(h,h)\right)dV\\
=\int_\Sigma\abs{k}^2dV_\gamma\int_{0}^{\infty} \left(f'^2-2(n-2)\tfrac{f^2}{r^2}\right)r^{n-1}dr,
\end{multline}
and this expression can become negative if and only if $n<10$.
\end{proof}

\section{The Ricci-flat cone over $\CC P^2$}\label{seccpt}

In this section we prove the following theorem:

\begin{theorem}\label{thmstab}
The Ricci-flat cone over $\CC P^2$ is stable. In particular, five is the smallest dimension of a stable Ricci-flat cone that is not flat.
\end{theorem}

For the proof of Theorem \ref{thmstab} we will use the following inequality:

\begin{theorem}[Warner \cite{Wa}]\label{thmwarner}
On $\CC P^2$ with the standard metric we have:
\begin{equation}
\int_{\CC P^2}\mathrm{tr}k\,dV=0\quad\Rightarrow\quad\int_{\CC P^2}\left[\abs{\D k}^2-2\Riem(k,k)\right]dV\geq 0.
\end{equation}
\end{theorem}

\begin{proof}[Proof of Theorem \ref{thmstab}]
Write $g=dr^2+r^2\gamma$ as usual. A general variation $h$ can be expanded as follows:
\begin{equation}
h=Adr^2+rB_i(dr\otimes dx^i+dx^i\otimes dr)+r^2C_{ij}dx^i\otimes dx^j.
\end{equation}
We assume that $h$ is transverse, i.e. $\text{div}_{g}(h)=0$. With respect to the $\gamma$-metric this transversality is expressed by the following equations:
\begin{eqnarray}
0&=&r\partial_rA+4A+\div(B)-\text{tr}(C),\label{div0}\\
0&=&r\partial_rB+5B+\div(C).\label{divi}
\end{eqnarray}
These relations between $A$, $B$ and $C$ will be used later. Next, we note that
\begin{equation}\label{riemcone}
\Riem_g(h,h)=\tfrac{1}{r^2}\left[\Riem_{\gamma}(C,C)+|C|^2-\text{tr}(C)^2\right],
\end{equation}
where the right hand side is computed with respect to the metric $\gamma$. Furthermore, a somewhat cumbersome computation yields:
\begin{align}
&|\nabla^g h|_{g}^2=(\partial_rA)^2+2|\partial_rB|^2+|\partial_rC|^2\\
&\quad+\tfrac{1}{r^2}\left[|\nabla A-2B|^2
+2|A\gamma+\nabla B-C|^2+|\nabla_i C_{jk}+\gamma_{ij} B_k+\gamma_{ik} B_j|^2\right].\nonumber
\end{align}
Squaring this out, and using also integration by parts and (\ref{riemcone}) we obtain
\begin{align}\label{hugeformula}
&\int_{C(\CC P^2 )} \left(|\nabla^g h|_{g}^2-2\Riem_g(h,h)\right)dV_g\\
&=\int_0^\infty\int_{\CC P^2}\big( (\partial_rA)^2+2|\partial_rB|^2+|\partial_rC|^2+\tfrac{1}{r^2}\big[|\nabla C|^2-2\Riem_{\gamma}(C,C)+2\mathrm{tr}(C)^2\nonumber\\
&\quad\qquad\qquad+|\nabla A|^2+2|\nabla B|^2+8A^2+14|B|^2-4A\mathrm{tr}(C)
+(8-\alpha)A\div(B)\nonumber\\
&\quad\qquad\qquad+(8-\beta)\langle  B,\div (C) \rangle-\alpha\langle \nabla A, B\rangle-\beta\langle \nabla B, C \rangle\big]\big)dVr^{4}dr,\nonumber
\end{align}
where $\alpha$ and $\beta$ are parameters that will be chosen later. Now, let us estimate the quantities in (\ref{hugeformula}) in five steps. First, by Kato's and Hardy's inequality we get
\begin{multline}
\int_0^\infty\int_{\CC P^2}\big[ (\partial_rA)^2+2|\partial_rB|^2+|\partial_rC|^2\big]dVr^{4}dr\\
\geq \int_0^\infty\int_{\CC P^2}\tfrac{1}{r^2}\big[ \tfrac{9}{4}A^2+\tfrac{9}{2}|B|^2+\tfrac{9}{4}|C|^2\big]dVr^4dr.
\end{multline}
Second, by Theorem \ref{thmwarner} applied to $C-\left(\tfrac{\int_{\CC P^2}\mathrm{tr}C\,dV}{4\int_{\CC P^2} dV}\right)\gamma$ and H\"older's inequality we obtain
\begin{align}
\int_0^\infty\int_{\CC P^2}\tfrac{1}{r^2}\big[|\nabla C|^2-2\Riem_{\gamma}(C,C)+\tfrac{3}{2}\mathrm{tr}(C)^2\big]dVr^{4}dr\geq 0.
\end{align}
Third, there is one term that we keep as it stands:
\begin{align}\label{thirdlarge}
\int_0^\infty\int_{\CC P^2}\tfrac{1}{r^2}\big[\tfrac{1}{2}\mathrm{tr}(C)^2+|\nabla A|^2+2|\nabla B|^2+8A^2+14|B|^2-4A\mathrm{tr}(C)\big]dVr^4dr.
\end{align}
Fourth, using (\ref{div0}), (\ref{divi}), and integration by parts we obtain
\begin{align}\label{fourthlarge}
&\int_0^\infty\int_{\CC P^2}\tfrac{1}{r^2}\big[
(8-\alpha)A\div(B)+(8-\beta)\langle  B,\div (C) \rangle\big]dVr^{4}dr\\
&\quad\geq\int_0^\infty\int_{\CC P^2}\tfrac{1}{r^2}\big[
-\tfrac{5}{2}(8-\alpha)A^2-\tfrac{7}{2}(8-\beta)\abs{B}^2+(8-\alpha)A\mathrm{tr}(C)\big]dVr^{4}dr.\nonumber
\end{align}
Fifth and finally, using the Cauchy-Schwarz inequality and Young's inequality we get
\begin{align}
&\int_0^\infty\int_{\CC P^2}\tfrac{1}{r^2}\big[
-\alpha\langle \nabla A, B\rangle-\beta\langle \nabla B, C \rangle\big]dVr^{4}dr\\
&\qquad\geq\int_0^\infty\int_{\CC P^2}\tfrac{1}{r^2}\big[
-\abs{\nabla A}^2-\tfrac{\alpha^2}{4}\abs{B}^2-2\abs{\nabla B}^2-\tfrac{\beta^2}{8}\abs{C}^2\big]dVr^{4}dr.\nonumber
\end{align}
Putting everything together, noting also that the sum of the cross-terms in (\ref{thirdlarge}) and (\ref{fourthlarge}) can be estimated by Young's inequality,
\begin{equation}
(4-\alpha)A\mathrm{tr}C\geq-\eps\abs{4-\alpha}\mathrm{tr}(C)^2-\tfrac{1}{4\eps}\abs{4-\alpha}A^2,
\end{equation}
we obtain the estimate
\begin{align}
&\int_{C(\CC P^2 )} \left(|\nabla^g h|_{g}^2-2\Riem_g(h,h)\right)dV_g\\
&\geq\int_0^\infty\int_{\CC P^2}\tfrac{1}{4r^2}\big[ (10\alpha-39-\abs{4-\alpha}/\eps)A^2+(-38-\alpha^2+14\beta)\abs{B}^2\nonumber\\
&\quad\qquad\qquad\qquad+(9-\beta^2/2)\abs{C}^2+(2-4\eps\abs{4-\alpha})\mathrm{tr}(C)^2\big]dVr^{4}dr.\nonumber
\end{align}
Choosing $\alpha$, $\beta$, and $\eps$ suitably (e.g. $\alpha=21/5$, $\beta=4$, $\eps=1$ does the job), we conclude that
\begin{align}
&\int_{C(\CC P^2 )} \left(|\nabla^g h|_{g}^2-2\Riem_g(h,h)\right)dV_g\geq 0,
\end{align}
for all transverse variations $h$ with compact support, i.e. the Ricci-flat cone over $\CC P^2$ is stable.
\end{proof}

\section{Ricci-flat cones over four-manifolds}\label{secff}
In the last section we have seen that the cone over $\CC P^2$ is stable. Besides this cone, we have the following fundamental examples of Ricci-flat cones over $4$-manifolds:
\begin{equation*}
C(S^2\times S^2),\;C(\CC P^2\sharp p \overline{\CC P^2})_{3\leq p\leq 8},\;C(\CC P^2\sharp \overline{\CC P^2}),\;C(\CC P^2\sharp 2 \overline{\CC P^2}).
\end{equation*}
The cone over $S^2\times S^2$ is unstable by Theorem \ref{produnstabthm}. Concerning the next family of examples, Tian proved the existence of K\"{a}hler-Einstein metrics on $\CC P^2\sharp p\overline{\CC P^2}$, that is on the blowup of  $\CC P^2$ at $3\leq p\leq 8$ points in general position \cite{Ti}. The cones over them are unstable by Theorem \ref{kaehlerthm}. Finally, there exists no K\"{a}hler-Einstein metric on $\CC P^2\sharp\overline{\CC P^2}$ and $\CC P^2\sharp 2\overline{\CC P^2}$, but an Einstein metric conformal to an extremal K\"{a}hler metric:

\begin{theorem}[Page \cite{Pa}, Chen-LeBrun-Weber \cite{CLW}]
There exists a positive Einstein metric on $\CC P^2\sharp \overline{\CC P^2}$ and $\CC P^2\sharp 2\overline{\CC P^2}$. This Einstein metric $g$ is conformal to an extremal K\"{a}hler metric $k$ with positive scalar curvature $s_k$, in fact $g=s_k^{-2}k$.
\end{theorem}

We will now investigate the stability of the Ricci-flat cone over the Page metric and the Chen-LeBrun-Weber metric. As a first step we express our stability integrand in terms of the conformally related extremal K\"{a}hler-metric, namely we have the following lemma:

\begin{lemma} For every traceless symmetric $2$-tensor $h$ we have the pointwise identity
\begin{align}
&\left(-\abs{\D^g h}_g^2+2\Riem_g(h,h)\right)dV_g=\\
&\quad \bigl(-\abs{\D h}^2+2\Riem(h,h)-4\abs{h}^2\abs{\D\log s}^2-2\langle\D\abs{h}^2,\D\log s\rangle-12\abs{h(\D\log s,\cdot)}^2
\nonumber\\
&\quad-4\langle {h}^2,\D^2\log s\rangle-4\langle \div{h}^2,\D\log s\rangle
+8h(\div h,\D\log s) \bigr)s^2dV,\nonumber
\end{align}
where the quantities on the right hand side are computed with respect to the extremal K\"{a}hler-metric $k$.
\end{lemma}

\begin{proof}
Since $g=s_k^{-2}k$, by the usual formulas for the conformal transformation of geometric quantities (see e.g. \cite[Thm. 1.159]{Be}) we obtain
\begin{align}
&\!\!\!\!\!\!\Riem_g(h,h)s_k^{-2}\tfrac{dV_g}{dV_k}\\
&=\left(\Riem_k+k\owedge(\mathrm{Hess}_k\log s+\D\log s\otimes \D \log s -\tfrac{1}{2}\abs{\D \log s}_k^2k)\right)(h,h)\nonumber\\
&=\Riem(h,h)-2\langle {h}^2,\D^2\log s\rangle-2\langle {h}^2,\D\log s\otimes \D\log s\rangle+\abs{h}^2\abs{\D\log s}^2,\nonumber
\end{align}
where the last line is computed with respect to the metric $k$. Furthermore
\begin{align}
&\!\!\!\!\!\!\abs{\D^g h}_g^2s_k^{-2}\tfrac{dV_g}{dV_k}=\abs{\D^g h}_k^2\\
&=\abs{\D_ih_{jk}+2h_{jk}\D_i\log s +h_{ik}\D_j\log s \nonumber\\
&\qquad+h_{ij}\D_k\log s -h_{pk}k_{ij}\D_p\log s-h_{pj}k_{ik}\D_p\log s}^2\nonumber\\
&=\abs{\D h}^2+6\abs{h}^2\abs{\D\log s}^2+8\abs{h(\D\log s,\cdot)}^2+4\D_ih_{jk}h_{jk}\D_i\log s \nonumber\\
&\qquad+4\D_ih_{jk}h_{ik}\D_j\log s -4\D_jh_{jk}h_{pk}\D_p\log s ,
\nonumber
\end{align}
and the claim follows.
\end{proof}

Let $\Sigma=\CC P^2\sharp \overline{\CC P^2}$ respectively $\CC P^2\sharp 2\overline{\CC P^2}$, and write $\tilde{h}=s^{-2}h$. From now on, we assume in addition $\div_g(h)=0$, that is $\div(s^{-2}h)=0$ with respect to the metric $k$. Using this, the lemma, and integration by parts we obtain
\begin{align}
&\int_\Sigma\left(-\abs{\D^g h}_g^2+2\Riem_g(h,h)\right)dV_g\\
&=\int_\Sigma \biggl(-\abs{\D h}^2+2\Riem(h,h)+2\abs{h}^2\Lap\log s+12\abs{h(\D\log s,\cdot)}^2  \biggr)s^2dV\nonumber\\
&=\int_\Sigma \biggl(-\abs{\D \tilde{h}}^2+2\Riem(\tilde{h},\tilde{h})+4\abs{\tilde{h}}^2\Lap\log s +8\abs{\tilde{h}}^2\abs{\D\log s}^2+12\abs{\tilde{h}(\D\log s,\cdot)}^2 \biggr)s^6dV.\nonumber
\end{align}
where again the right hand side is computed with respect to the metric $k$.
Since $h^{(1,1)}(\Sigma )>1$, we can find a traceless test variation $\tilde{h}$ that comes from a harmonic $(1,1)$-form on the K\"ahler manifold $(\Sigma ,k)$. Then the Bochner formula gives
\begin{equation}
0=\Lap\tilde{h}+2\Riem(\tilde{h},\cdot)-\Rc.\tilde{h}-\tilde{h}.\Rc.
\end{equation}
Furthermore, the conformal transformation law for the Ricci-tensor yields
\begin{equation}
\Rc=\left(3s^{-2}+2\abs{\D\log s}^2-\Lap\log s\right)k-2\left(\D^2\log s+\D\log s \otimes \D\log s\right).
\end{equation}
Finally, a pointwise computation gives the following formulas:
\begin{align}
4\abs{\tilde{h}(\D\log s,\cdot)}^2&=\abs{\tilde{h}}^2\abs{\D\log s}^2,\\
4\div(\tilde{h}^2)&=\D\abs{\tilde{h}}^2.
\end{align}
Putting everything together we obtain the following theorem (see also \cite{SHthesis} for an alternative proof):
\begin{theorem} If $h$ is a test-variation as above, then:
\begin{align}
&\int_\Sigma \left(-\abs{\D^g h}_g^2+2\Riem_g(h,h)\right)dV_g=\int_\Sigma \left(6-\Lap_k s^2\right)\abs{h}_g^2dV_g.
\end{align}
\end{theorem}

\begin{proof}
Putting together the above facts we compute
\begin{align}
&\!\!\!\int_\Sigma \left(-\abs{\D^g h}_g^2+2\Riem_g(h,h)\right)dV_g\\
&=\int_\Sigma \biggl(\langle\tilde{h},\Lap\tilde{h}\rangle+2\Riem(\tilde{h},\tilde{h})+\abs{\tilde{h}}^2\Lap\log s -7\abs{\tilde{h}}^2\abs{\D\log s}^2 \biggr)s^6dV\nonumber\\
&=\int_\Sigma \biggl(6s^{-2}\abs{\tilde{h}}^2-\abs{\tilde{h}}^2\Lap\log s -4\langle\tilde{h}^2,\D^2\log s\rangle -4\abs{\tilde{h}}^2\abs{\D\log s}^2\biggr)s^6dV\nonumber\\
&=\int_\Sigma \biggl(6s^{-2}\abs{\tilde{h}}^2-2\abs{\tilde{h}}^2\Lap\log s-4\abs{\tilde{h}}^2\abs{\D\log s}^2\biggr)s^6dV\nonumber,
\end{align}
and the claim follows.
\end{proof}

To finish the proof of Theorem \ref{thm5} and Conjecture \ref{conj2} it would be sufficient to prove the pointwise inequality $\Lap s^2<6$ for the extremal K\"ahler metrics conformal to the Page-metric respectively to the Chen-LeBrun-Weber metric (normalized such that the Einstein constant equals $3$). To finish the proof of Theorem \ref{thm4} and Conjecture \ref{conj1} an estimate $\Lap s^2<15/4$ would be sufficient, as shown by the following lemma.

\begin{lemma}
Let $C(\Sigma,\gamma)$ be a Ricci-flat cone over a four manifold and assume there exists a transverse-traceless symmetric 2-tensor $h$ on $\Sigma$ with
\begin{equation}
\int_{\Sigma}\left(-\abs{\D h}^2+2\Riem(h,h)-\tfrac{9}{4}\abs{h}^2\right)dV>0.
\end{equation}
Then the Ricci-flat cone $C(\Sigma,\gamma)$ is unstable.
\end{lemma}

\begin{proof}
Consider the  variation
\begin{equation}
H=f(r)r^2h.
\end{equation}
This variation is TT, and we compute
\begin{align}
&\abs{\nabla H}^2=\left(f'^2+2\tfrac{f^2}{r^2}\right)\abs{h}^2+\tfrac{f^2}{r^2}\abs{\nabla h}^2\\
&2\Riem(H,H)=2\tfrac{f^2}{r^2}\left(\Riem_\gamma(h,h)+\abs{h}^2\right).
\end{align}
By compactness the assumption of the lemma is also satisfied for some $\lambda$ strictly greater than $\tfrac{9}{4}$. Thus
\begin{equation}
\int_{C(\Sigma)}\left(|\nabla H|^2-2\text{Rm}(H,H)\right)dV
<\int_\Sigma\abs{h}^2dV_\gamma\int_{0}^{\infty} \left(f'^2-\lambda\tfrac{f^2}{r^2}\right)r^{n-1}dr.
\end{equation}
Since the Hardy-constant is $C_H=\tfrac{4}{9}$ for $n=5$, we can choose $f$ such that this expression becomes negative.
\end{proof}

\section{Estimates for extremal K\"ahler metrics}\label{extremalk}
The purpose of this section is to estimate $\Lap s^2$ for the extremal metric corresponding to the Page metric respectively to the Chen-LeBrun-Weber metric. Some of these estimates also appeared in \cite{SHthesis}. We will prove:
\begin{theorem}\label{Pageest}
Let $k$ be the extremal K\"ahler metric on $\CC P^2\sharp \overline{\CC P^2}$, such that $s^{-2}k$ is the Page metric with Einstein-constant equal to $3$. Then we have the pointwise estimate
\begin{equation}
\Lap s^2<15/4.
\end{equation}
\end{theorem}
Furthermore, we will give strong numerical evidence for:
\begin{conj}\label{CLWest}
Let $k$ be the extremal K\"ahler metric on $\CC P^2\sharp 2\overline{\CC P^2}$, such that $s^{-2}k$ is the Chen-LeBrun-Weber metric with Einstein-constant equal to $3$. Then we have the pointwise estimate
\begin{equation}
\Lap s^2<15/4.
\end{equation}
\end{conj}

Note that, using the results from the previous sections, Theorem \ref{Pageest} finishes the proofs of Theorem \ref{thm4} and Theorem \ref{thm5}, and the statement of Conjecture \ref{CLWest} implies the statements of Conjecture \ref{conj1} and Conjecture \ref{conj2}.\\

%%\begin{theorem}[?]
%Let $k$ be the extremal K\"ahler metric on $\CC P^2\sharp \overline{\CC 2P^2}$, such that $s^{-2}k$ is the Chen-LeBrun-Weber metric with Einstein-constant equal to $3$. Then we have the pointwise estimate
%\begin{equation}
%\Lap s^2<15/4.
%\end{equation}
%\end{theorem}

%Before we begin we should say that the instability of the Page metric in Theorem \ref{thm5}  is an old result of Roberta Young \cite{Young1983}. By exploiting the $U(2)$ symmetry of the Page metric she was able to numerically solve the eigenvalue problem
%$$-\Delta h+2 Rm(h,\cdot) = \lambda h.$$
%She exhibited an eigentensor $h_{Y}$ with positive eigenvalue $\lambda$ satisfying
%$$  1.773\Lambda  \leq \lambda \leq 1.776\Lambda $$
%where the Page metric solves $Ric(g)=\Lambda g$. Hence if we scale so that $\Lambda =3$ we have
%$$\int_{\Sigma}\left(-\abs{\D h_{Y}}^2+2\Riem(h_{Y},h_{Y})-\tfrac{9}{4}\abs{h_{Y}}^2\right)dV \geq 3.069 \int_{\Sigma} \abs{h_{Y}}^2 dV>0 $$
%and the instability of the cone $C(\CC P^{2}\sharp\overline{\CC P}^{2})$ follows.

To prove Theorem \ref{Pageest} and to give evidence for Conjecture \ref{CLWest} we will employ the fact that the extremal metrics on both $\CC P^{2}\sharp\overline{\CC P^{2}}$ and $\CC P^{2}\sharp 2 \overline{\CC P^{2}}$ are toric-K\"ahler metrics. The theory of toric-K\"ahler metrics is well developed and documented \cite{Abr2}. The main aspect we will use is that a toric-K\"aher manifold $\Sigma^{2m}$ (in our case $m=2$) admits a dense open set $\Sigma^{\circ}$ on which the action of the torus is free. Furthermore one may pick a special coordinate system called \emph{symplectic coordinates} such that
$$\Sigma^{\circ}  = P^{\circ}\times\mathbb{T}^{m}$$
where $P^{\circ}$ is the interior of a convex polytope $P \subset \mathbb{R}^{m}$  known as the \emph{moment polytope}. In these coordinates the metric is encoded by a convex function $u:P^{\circ} \rightarrow \mathbb{R}$ (known as the \emph{symplectic potential}) in the following way:
$$k = u_{ij}dx^{i}dx^{j}+u^{ij}d\theta^{i}d\theta^{j}.$$
Here the $x^{i}$ are coordinates on the polytope $P$, $\theta^{i}$ are coordinates on the torus, $u_{ij}$ is the Euclidean Hessian of $u$ and $u^{ij}$ is the matrix inverse of $u_{ij}$.\\

In general, an $r$-sided polytope (our polytope $P$ is in fact a \emph{Delzant} polytope \cite{Del}), can be described by $r$ inequalities $l_{i}(x)>0$ where the $l_{i}(x)$ are affine functions of $x$. As the scalar curvature is invariant under the torus action it is a function of $x$ only. The equations for an extremal metric in these coordinates can be shown to be equivalent to requiring that the scalar curvature is an affine function of the polytope coordinates, i.e.
$$s(x_{1},\ldots,x_m) = \sum_{i=1}^ma_ix_i+b.$$
The constants $a_i$ and $b$ can all be worked out \emph{a priori} from the elementary geometry of the polytope. In order to describe how, for each $l_{i}$ we define a one-form $d\sigma_{i}$  by requiring,
$$ dl_{i} \wedge d\sigma_{i} = \pm dx,$$
where the $l_{i}$ are the functions from above and $dx$ denotes the Euclidean volume form. The one-form $d\sigma_{i}$ then defines a measure on the edge defined by the $l_{i}$. We denote the measure obtained on the whole boundary $\partial P$ in this way by $d\sigma$. Donaldson proves the following integration by parts formula:

\begin{lemma}[Donaldson \cite{DonInt}]\label{LDonIBP}
Let $(M,k)$ be a toric-K\"ahler manifold with polytope $P$, symplectic potential $u$ and let $s_{k}$ be the scalar curvature of $k$. Then for all $f \in C^{\infty}(P)$
\begin{equation}\label{DonIBP}
\int_{P}u^{ij}f_{ij}dx = \int_{\partial P}fd\sigma-\int_{P}s_{k}fdx,
\end{equation}
where $d\sigma$ is the measure defined above.  
\end{lemma}
Clearly the left-hand-side of (\ref{DonIBP}) vanishes for all affine functions $f$. Hence if $s_k = a_{1}x_{1}+...a_{m}x_{m}+b$ then by successively taking  $f=x_{1}$ to $f=x_{m}$ and $f=1$ we obtain $m+1$ linear constraints for the $m+1$ unknowns and we can find the $a_{i}$ and $b$ explicitly. A good reference for this is the survey by Donaldson \cite{Don}.\\

Another point to note is that we do not really need to compute two derivatives of the scalar curvature. If $g=s^{-2}k$ has constant scalar curvature $\kappa$, then the conformal transformation law for the scalar curvature in dimension four, gives the equation
\begin{equation}
s^{3}+6s\Delta s-12|\nabla s|^{2} = \kappa.
\end{equation}
Now as $\Delta s^{2} = 2s \Delta s+2|\nabla s|^{2}$  we have
\begin{equation}
\Delta s^{2} = \tfrac{\kappa}{3}+6|\nabla s|^{2}-\frac{s^{3}}{3}.
\end{equation}
Finally, a word about scaling. The extremal K\"ahler metrics below will show up with some specific normalization. Let $K:=\sup\Lap s^2$ in this normalization.
A rescaling of the K\"ahler metric $k\rightarrow c^{2}k$ gives a rescaling $s \rightarrow c^{-2}s$ of the scalar curvature and hence a rescaling $g\rightarrow c^{6}g$ of the Einstein metric. Now
$$\Rc(c^{6}g)=\Rc(g) = \tfrac{\kappa}{4} g = \frac{\kappa}{4c^{6}}c^{6}g,$$  
so we rescale the K\"ahler metric by $c^{2}$ where $c=\left(\frac{\kappa}{12}\right)^{1/6}$ to ensure that the Einstein constant is $3$. Since $\Lap s^2\rightarrow c^{-6}\Lap s^2$, for this correctly normalized metric we have
\begin{align}
\sup\Delta s^{2}=\tfrac{12}{\kappa}K,
\end{align}
i.e. what we want to show is the inequality $\tfrac{12}{\kappa}K<\tfrac{15}{4}$.

\subsection{The Page Metric}
The moment polytope of the manifold $\CC P^{2}\sharp\overline{\CC P}^{2}$ is a trapezium $T$.  We shall parameterise it as the set of points $x=(x_1,x_2)\in\RR^2$ satisfying the inequalities $l_{i}(x)>0$ where:
$$l_{1}(x) = x_{1}, l_{2}(x) = x_{2},l_{3}(x)= 1-x_{1}-x_{2} \text{ and } l_{4}(x) = x_{1}+ x_{2}-a.$$
Here $a$ is a parameter that effectively determines the K\"ahler class by varying the volume of the exceptional divisor. Abreu \cite{Abr1} following Calabi \cite{Cal} gave an explicit description of the symplectic potential $u$ of the extremal metric $k$. It has the form:
$$u(x_{1},x_{2}) = \sum_{i}l_{i}(x)\log(l_{i}(x)) +f(x_{1}+x_{2})$$
where $l_{i}$ are the lines defining the trapezium and the function $f$ satisfies
$$f^{''}(t) = \frac{2a(1-a)}{2at^{2}+(1+2a-a^{2})t+2a^{2}}-\frac{1}{t}.$$
The scalar curvature is then an affine function of $x_{1}+x_{2}$ and is given by $s(x_{1},x_{2}) =c_{1}(x_{1}+x_{2})+c_{2} $ where
$$c_{1} = \frac{24 a}{(1-a)(1+4a+a^{2})} \text{ and } c_{2} = \frac{6(1-3a^{2})}{(1-a)(1+4a+a^{2})}.$$\\
A computation aided by Mathematica (we emphasize that this and all the following computations involve just polynomials and in principle could be carried out to arbitrary precision without computer aid) shows that the scalar curvature $\kappa$ of the Page metric $g=s^{-2}k$, that is $\kappa=s^{3}+6s\Delta s-12|\nabla s|^{2}$, equals
\begin{tiny} 
\begin{equation}\label{Scalpage}
\frac{864a^{2}(1-6a^{2}-16a^{3}+9a^{4})+216(3a^{6}-24a^{5}+53a^{4}+32a^{3}-15a^{2}-8a-1)(x_{1}+x_{2}
)}{(a-1)^{3}(1+4a+a^{2})^{3}(x_1+x_2)}.
\end{equation}
\end{tiny}

Hence the extremal metric occurs in the class where $0<a<1$ and
$$1-6a^{2}-16a^{3}+9a^{4}=0.$$
This value is $a\approx 0.31408$. Plugging this into \ref{Scalpage} yields the scalar curvature of the Page metric $\kappa\approx 182.219$.\\

As an aside, LeBrun \cite{LeB} calculates that the critical K\"ahler class is the one for which the area of a projective line is 3.1839 times the area of the exceptional divisor. In our toric description the area of a projective line has been scaled to be 1 and the area of the exceptional divisor is $a$, hence $a\approx(3.1839)^{-1}\approx 0.31408$. Another neat verification one can do is to compute the volume of the Page metric using the toric description.
This is given by
$$(4\pi)^{2}\int \int_{T}(c_1(x_1+x_2)+c_2)^{-4}dx_1dx_2 \approx 0.072699.$$
The factor $(4\pi)^{2}$ appears above as the $S^{1}$ factors have volume $4\pi$ in this description. In Page's original paper \cite{Pa} he gives the volume $V$ of the metric as $V = 150.862\Lambda^{-2}$ where $\Lambda$ is the Einstein constant. Hence in our case $\Lambda \approx 45.554$ and so $\kappa\approx 182.219.$\\

Getting back to the original problem, a calculation (aided by Mathematica) shows that
\begin{align}
\Lap s^2&= \frac{\kappa}{3}+6|\nabla s|^{2}-\frac{s^{3}}{3}=\frac{\kappa}{3}+\frac{1}{(a-1)^{3}(1+4a+a^{2})^{3}t}\sum_{i=0}^4\alpha_i t^i,
\end{align}
with $t=x_1+x_2$ and the coefficients $\alpha_i$ given by:
\begin{align*}
\alpha_0&=6912a^{5}\\
\alpha_1&=72-648a^{2}+3456a^{3}+1944a^{4}-10368a^{5}-1944a^{6}\\
\alpha_2&=864a-3456a^{2}-15552a^{3}+10368a^{4}+11232a^{5}\\
\alpha_3&=6912a^{2}-20736a^{5}\\
\alpha_4&=11520a^{3}.
\end{align*}
Denote by $K$ the supremum of $\Lap s^2$ over $T$. It can be checked that $\tfrac{12}{\kappa}K<\tfrac{15}{4}$. In fact, a calculation (aided by Mathematica) shows that $\tfrac{12}{\kappa}K < 2.65$. This proves Theorem \ref{Pageest}.
\begin{center}
\end{center}

\subsection{The Chen-LeBrun-Weber metric}

The moment polytope for $\Sigma = \CC P^{2} \sharp 2 \overline{\CC P}^{2}$ is a pentagon $P$ with vertices at $(0,0), (a,0), (a,1), (1,a)$ and $(0,a)$.  If we view $\CC P^{2} \sharp 2 \overline{\CC P}^{2}$ as $(\CC P^{1}\times \CC P^{1}) \sharp \overline{\CC P}^{2}$, then $a$ effectively determines the volume of the exceptional divisor. Chen, LeBrun and Weber calculate that their extremal metric $k$ occurs for $a \approx 1.958$ \cite{CLW}. In order to calculate the scalar curvature using Lemma \ref{LDonIBP} we compute the integrals
$$A=\int_{P}x_{1}^{2}dx = \int_{P}x_{2}^{2}dx = \frac{1}{12}(a^{4}+4a^{3}-1),$$
$$B = \int_{P}x_{1}x_{2}dx = \frac{1}{24}(a^{4}+4a^{3}+6a^{2}-4a-1),$$
$$C =\int_{P}x_{1}dx = \int_{P}x_{2}dx = \frac{1}{6}(a^{3}+3a^{2}-1),$$
$$D=\int_{P}dx = \frac{1}{2}(a^{2}+2a-1),$$
and
$$E_{0} = \int_{\partial P}d\sigma =1+3a,\qquad E_{1} = \int_{\partial P} x_{1} d\sigma = \int_{\partial P}x_{2} d\sigma = a^{2}+a.$$
Hence, if $s_{k} = a_{1}x_{1}+a_{2}x_{2}+b$ then we can find $a_{i}$ and $b$ by solving
$$\left( \begin{array}{ccc}
A & B & C \\
B & A & C \\
C & C & D
\end{array}\right)\left(\begin{array}{c}
a_{1} \\ a_{2} \\ b
\end{array}\right) = \left(\begin{array}{c}
E_{1} \\ E_{1} \\ E_{0}
\end{array}\right).$$
Solving this system the value of $a=1.958$ yields the scalar curvature 
\begin{equation}
s_{k}(x_{1},x_{2}) = -0.423(x_{1}+x_{2})+2.790.
\end{equation}

Let $g=s_{k}^{-2}k$ be the Chen-LeBrun-Weber metric. We will now determine the normalization. Since $\Rc(g)=\Lambda g$, the Gauss-Bonnet formula gives
$$\chi(\Sigma)=\frac{1}{8\pi^2}\int_{\Sigma}\left(\abs{W_g}^2-\frac{\abs{\Rc_g^\circ}}{2}+\frac{R_g^2}{24}\right)dV_g=\frac{1}{8\pi^2}\int_{\Sigma}\abs{W_g}^2dV_g+\frac{\Lambda^2}{12\pi^2}\mathrm{Vol}_g(\Sigma).$$
The integral $\int_{\Sigma}\abs{W_g}^2dV_g$ is conformally invariant, and thus can be computed with respect to the K\"ahler metric $k$. The Hirzebruch signature formula reads
$$\sigma(\Sigma)=\frac{1}{12\pi^2}\int_{\Sigma}\left(\abs{W_+}^2-\abs{W_-}^2\right)dV,$$
and it is a standard fact for K\"ahler surfaces that $\abs{W_+}^2=s^2/24$. Putting things together we obtain
$$\int_{\Sigma}\abs{W}^2dV=2\int_{\Sigma}\abs{W_+}^2dV-12\pi^2\sigma(\Sigma)=\frac{1}{12}\int_{\Sigma}s^2dV-12\pi^2\sigma(\Sigma),$$
and conclude that
$$\Lambda^{2} = \frac{96\pi^{2}\chi(\Sigma)+144\pi^{2}\sigma(\Sigma)-\int_{\Sigma}s^{2}dV}{8\textrm{Vol}_g(\Sigma)}.$$
As $\textrm{Vol}_g(\Sigma) = 16\pi^{2}\int_{P}s(x)^{-4}dx$ we can calculate $ \Lambda \approx 1.886 $, $\kappa\approx 7.54$.\\

Our task is now to estimate $K:=\sup\Lap s^2$. To do this, we will use Donaldson's method of numerically approximating extremal metrics by `balanced' metrics \cite{Donum}, implemented by Donaldson and Bunch in \cite{DonBun}. Donaldson's algorithm only works for rational K\"ahler classes (and is computationally unfeasible for rational numbers with large denominators). It is a wonderful serendipity that the value $a \approx 1.958$ is actually very close to being integral and determining an integral K\"ahler class (unlike for the Page metric). The K\"ahler class corresponding to $a=2$ is in fact the anticanonical class $c_{1}(\Sigma)$. Using this approximation, Donaldson's algorithm gives the estimate $K<1.363$,\footnote{The C++ files are available on http://www.buckingham.ac.uk/directory/dr-stuart-hall/} and thus $\tfrac{12}{\kappa}K<2.17$. This gives strong evidence for Conjecture \ref{CLWest}.
 
\section{Ilmanen's conjecture}\label{ilmanenconj}
Let us first explain the relevant background about Perelman's energy functional, and our new variant for the noncompact case: As Ilmanen pointed out, the functional
\begin{equation}
\lambda(g)=\inf_{w:\int \! w^2=1}\int(4\abs{\D w}^2+Rw^2)dV
\end{equation}
does not detect counterexamples to the positive mass theorem for ALE-spaces, i.e. asymptotically locally euclidean metrics with nonnegative scalar curvature and negative mass. Here, the mass of an ALE-space of order $\tau>(n-2)/2$ is defined as
\begin{equation}\label{admmass}
m(g)=\lim_{r\to\infty}\int_{S_r}\left(\partial_jg_{ij}-\partial_ig_{jj}\right)dA^{i}.
\end{equation}
Minimizing sequences $w_i$ for $\lambda$ escape to infinity, giving $\lambda(g)=0$. The solution is to consider instead a noncompact variant of Perelman's energy functional that the second author introduced in \cite{Ha2}. For ALE-spaces, Ricci-flat cones or manifolds asymptotic to Ricci-flat cones it takes the form
\begin{equation}
\lambda_{\textrm{nc}}(g)=\inf_{w:\, w\to 1}\int(4\abs{\D w}^2+Rw^2)dV,
\end{equation}
where the infimum is now taken over all smooth functions $w$ approaching $1$ at infinity in the sense that $w-1=O(r^{-\tau})$ (here, the $O$-notation includes the condition that the derivatives decay appropriately), where $\tau>(n-2)/2$. In particular, $\lambda_{\textrm{nc}}(g)$ is strictly positive if the scalar curvature is nonnegative and positive at some point.\\

We can now formulate precisely and prove the first part of Ilmanen's conjecture:

\begin{theorem}
Let $M^n\, (n\geq 3)$ be a complete manifold with one end, and assume the end is diffeomorphic to $S^{n-1}\!/\Gamma\times\RR$.
Then the following are equivalent:
\begin{enumerate}
\item (lambda not a local maximum) There exists a metric $g$ on $M$ that agrees with the flat conical metric outside a compact set such that $\lambda_{\textrm{nc}}(g)>0$.
\item (failure of positive mass) There exists an asymptotically locally euclidean metric $g$ on $M$ such that $R_{g}\geq 0$ and $m(g)<0$.
\end{enumerate}\end{theorem}

\begin{proof}
$(1)\Rightarrow(2)$: Let $g$ be a metric on $M$ hat agrees with the flat conical metric outside a compact set such that $\lambda_{\textrm{nc}}(g)>0$. Note that
\begin{equation}\label{strpos}
\inf_{w:\, w\to 1}\int\left(\tfrac{4(n-1)}{(n-2)}\abs{\D w}^2+Rw^2\right)dV\geq\lambda_{\textrm{nc}}(g)>0.
\end{equation}
For each metric $g$ we have an elliptic index-zero operator
\begin{equation}
-\tfrac{4(n-1)}{(n-2)}\Lap_g+ R_g:C^{2,\alpha}_\tau(M)\rightarrow C^{0,\alpha}_{\tau+2}(M),
\end{equation}
between H\"older-spaces with weight $\tau=n-2-\eps$, $\eps>0$ small. By perturbing the metric a bit we can assume that this operator is invertible. Choosing the perturbation small enough we can also assume that (\ref{strpos}) still holds. Writing $w=1+u$ we can now solve the equation
\begin{equation}
\left(-\tfrac{4(n-1)}{(n-2)}\Lap_g+ R_g\right)w=0,\quad w\to 1\,\,\textrm{at}\,\,\infty.
\end{equation}
The conformally related metric $\tilde{g}:=w^{4/n-2}g$ has vanishing scalar curvature and is asymptotically locally euclidean of order $\tau$. Moreover, by a computation as in \cite{Ha2}, the mass drops down under the conformal rescaling, in fact
\begin{equation}
m(\tilde{g})=m(g)-\int\left(\tfrac{4(n-1)}{(n-2)}\abs{\D w}^2+Rw^2\right)dV.
\end{equation}
Using $m(g)=0$ and the inequality (\ref{strpos}) this implies $m(\tilde{g})<0$.\\

$(2)\Rightarrow(1)$: Let $g$ be an asymptotically locally euclidean metric on $M$ such that $R_{g}\geq 0$ and $m(g)<0$. By an ALE-version of Lohkamp's method \cite[Prop. 6.1]{Loh}, we can find a metric $\tilde{g}$ that agrees with the flat conical metric outside a compact set, such that $R_{\tilde{g}}\geq 0$ and $R_{\tilde{g}}>0$ somewhere. In particular, $\lambda_{\textrm{nc}}(\tilde{g})>0$.
\end{proof}

To discuss the second part of Ilmanen's conjecture, let $(C,g_C)$ be a Ricci-flat cone (of dimension at least $3$) over a closed manifold. In particular, this is a static (nonsmooth) solution of the Ricci flow. We say that an instantaneously smooth Ricci flow is coming out of the cone, if there exists a smooth  Ricci flow $(M,g(t))_{t\in(0,T)}$ such that the following conditions are satisfied:
\begin{itemize}
\item $(M,g(t))$ is complete with bounded curvature for each $t\in(0,T)$.
\item $(M,g(t))$ converges to $(C,g_C)$ for $t\to 0$ in the Gromov-Hausdorff sense everywhere and in the Cheeger-Gromov away from the tip.
\item There is no negative $L^1$-curvature concentration in the tip, in the sense that $\lim\inf_{t\to 0}\lambda_{\textrm{nc}}(g(t))\geq0$.
\end{itemize}

Having the static solution and another solution coming out of the cone in particular gives nonuniqueness of the Ricci flow with conical initial data. With the above definition we have:

\begin{theorem}
If there is an instantaneously smooth Ricci flow coming out of $(C,g_C)$ then there exists a complete smooth manifold $(M,g)$ that exponentially approaches $(C,g_C)$ at infinity such that $\lambda_{\textrm{nc}}(g)>0$. However, there exist unstable Ricci-flat cones with no instantaneously smooth Ricci flow coming out of them.
\end{theorem}

\begin{proof}
Let $(M,g(t))_{t\in(0,T)}$ be the Ricci flow coming out of the cone. Using Perelman's pseudolocality theorem we get a rough decay estimate, and using maximum principle estimates the decay rate can be improved. This is explained in detail in \cite{Siep}, and the conclusion is that $g(t)$ approaches $g_C$ exponentially at infinity. Using this decay estimate, by a similar computation as in \cite{Ha3}, we obtain the monotonicity formula,
\begin{equation}
\frac{d}{dt}\lambda_{\textrm{nc}}(g(t))=2\int_M\abs{\Rc+\D^2f}^2e^{-f}dV\geq 0,
\end{equation}
where $w=e^{-f/2}$ is the minimizer in the definition of $\lambda_{\textrm{nc}}$.
By the definition from above and the monotonicity of $\lambda_{\textrm{nc}}$ we obtain $\lambda_{\textrm{nc}}(g(t))>0$ for $t>0$ and $(M,g(t))$ has the desired properties.\\

For the final part, let $C=C(\Sigma)$ be an unstable Ricci-flat cone over closed 4-manifold $\Sigma$ that does not bound a smooth 5-manifold. From Section \ref{secff} we know that there are many such examples. Since $\Sigma$ does not bound a smooth 5-manifold, one cannot even find topologically a manifold $M$ smoothing out the cone. Thus there is no smooth Ricci flow coming out of the cone.
\end{proof}

In full generality, it is an interesting open problem if there exist singular Ricci flows coming out of unstable Ricci-flat cones. To answer it, one has to come up with a good notion of singular Ricci flow solutions first.

\end{document}